\documentclass{amsart}
\usepackage{amssymb, amsmath, latexsym, mathrsfs, verbatim, url, bbm, bbold}

\theoremstyle{plain}
\newtheorem{theorem}{Theorem}
\newtheorem{lemma}[theorem]{Lemma}
\newtheorem{proposition}[theorem]{Proposition}
\newtheorem{corollary}[theorem]{Corollary}

\theoremstyle{definition}

\newtheorem{question}[theorem]{Question}

\newcommand{\cl}{\mathsf{cl}}
\newcommand{\re}{\upharpoonright}
\newcommand{\dime}{\mathsf{dim}}
\newcommand{\Ind}{\mathsf{Ind}}
\newcommand{\dom}{\mathsf{dom}}
\newcommand{\ran}{\mathsf{ran}}
\newcommand{\supi}{\mathsf{sup}}
\newcommand{\ZFC}{\mathsf{ZFC}}
\newcommand{\Fin}{\mathsf{Fin}}
\newcommand{\Gd}{\mathsf{G_\delta}}
\newcommand{\BB}{\mathcal{B}}
\newcommand{\PPP}{\mathbb{P}}
\newcommand{\QQQ}{\mathbb{Q}}

\begin{document}

\title{Infinite powers and Cohen reals}

\author{Andrea Medini}
\address{Kurt G\"odel Research Center for Mathematical Logic
\newline\indent University of Vienna
\newline\indent W\"ahringer Stra{\ss}e 25
\newline\indent A-1090 Wien, Austria}
\email{andrea.medini@univie.ac.at}
\urladdr{http://www.logic.univie.ac.at/\~{}medinia2/}

\author{Jan van Mill}
\address{KdV Institute for Mathematics
\newline\indent University of Amsterdam
\newline\indent Science Park 904
\newline\indent P.O. Box 94248
\newline\indent 1090 GE Amsterdam, The Netherlands}
\email{j.vanMill@uva.nl}
\urladdr{http://staff.fnwi.uva.nl/j.vanmill/}

\author{Lyubomyr Zdomskyy}
\address{Kurt G\"odel Research Center for Mathematical Logic
\newline\indent University of Vienna
\newline\indent W\"ahringer Stra{\ss}e 25
\newline\indent A-1090 Wien, Austria}
\email{lyubomyr.zdomskyy@univie.ac.at}
\urladdr{http://www.logic.univie.ac.at/\~{}lzdomsky/}


\thanks{The first-listed author acknowledges the support of the FWF grant M 1851-N35. The second-listed author acknowledges generous hospitality and support from the Kurt G\"odel Research Center for Mathematical Logic. The third-listed author acknowledges the support of the FWF grants I 1209-N25 and I 2374-N35.}

\date{May 31, 2017}

\begin{abstract}
We give a consistent example of a zero-dimensional separable metrizable space $Z$ such that every homeomorphism of $Z^\omega$ acts like a permutation of the coordinates almost everywhere. Furthermore, this permutation varies continuously. This shows that a result of Dow and Pearl is sharp, and gives some insight into an open problem of Terada. Our example $Z$ is simply the set of $\omega_1$ Cohen reals, viewed as a subspace of $2^\omega$.
\end{abstract}

\maketitle

\section{Introduction}

A space $X$ is \emph{homogeneous} if for every pair $(x,y)$ of points of $X$ there exists a homeomorphism $f:X\longrightarrow X$ such that $f(x)=y$. This is a classical notion, which has been studied in depth (see for example \cite{arkhangelskiivanmill}).

It is an interesting theme in general topology that taking infinite powers tends to improve the homogeneity-type properties of a space. The first result of this kind is due to Keller, who showed that the Hilbert cube $[0,1]^\omega$ is homogeneous (see \cite{keller}). But this phenomenon is particularly striking in the zero-dimensional case, as Lawrence showed that $X^\omega$ is homogeneous for every separable metrizable zero-dimensional space $X$ (see \cite{lawrence}), answering a question of Fitzpatrick and Zhou from \cite{fitzpatrickzhou}. In fact, the following result (which answers a question of Gruenhage from \cite{gruenhage}) shows that this holds for a much wider class of spaces (see \cite[Theorem 3]{dowpearl}).

\begin{theorem}[Dow, Pearl]\label{dowpearlmain}
Let $X$ be a zero-dimensional first-countable space. Then $X^\omega$ is homogeneous.	
\end{theorem}

In order to better motivate our result, we need to dig a little deeper into the proof of Theorem \ref{dowpearlmain}. The first step of this proof is given by the following result (see \cite[Theorem 1]{dowpearl} and the subsequent remarks).\footnote{\,The second step consists in reducing the general case to the first step by using the technique of elementary submodels, but this will not be relevant for us.} Here, a \emph{partial permutation} of $\omega$ is a function of the form $h\re I$, where $h$ is a permutation of $\omega$ and $I\subseteq\omega$.

\begin{theorem}[Dow, Pearl]\label{dowpearlprecise}
If $D\subseteq 2^\omega$, $x,y\in D^\omega$, and $D\subseteq X\subseteq\cl(D)$, then there exist a homeomorphism $f:X^\omega\longrightarrow X^\omega$ and partial permutations $h_z$ of $\omega$ for $z\in X^\omega$ satisfying the following conditions:
\begin{itemize}
\item $f(x)=y$,
\item $\forall z\in X^\omega\,\forall i\in\dom(h_z)\,\exists U\ni z\text{ open in }X^\omega\,\forall w\in U\, (h_w(i)=h_z(i))$,
\item $\forall z\in X^\omega
\left\{
\begin{array}{ll} f(z)(i) = z(h_z(i)) & \textrm{if }i\in\dom(h_z),\\
f(z)(i) \in D & \textrm{if }i\notin\dom(h_z),\\
z(i) \in D & \textrm{if }i\notin\ran(h_z).\\
\end{array}
\right.
$
\end{itemize}
\end{theorem}

Notice that if $z$ belongs to the set
$$
C(D)=\bigcap_{d\in D,i\in\omega}\{w\in X^\omega:w(i)\neq d\}\cap\bigcap_{d\in D,i\in\omega}\{w\in X^\omega:f(w)(i)\neq d\},
$$
then $\dom(h_z)=\ran(h_z)=\omega$, so that $h_z$ will be a (full) permutation of the coordinates. Furthermore, a closer inspection of the construction of $f$ shows that there will always be a countable\footnote{\,In fact, $D'\subseteq\{\inf(D\cap[s]):s\in\Fin(\omega\times\omega,2)\}$, where the $\inf$ is taken with respect to an arbitrary well-order of $D$ fixed at the beginning of the proof.} $D'\subseteq D$ such that $f$ satisfies the above conditions for $D'$. Since $C(D')$ will be comeager, we conclude that for every homeomorphism $f$ produced by (the proof of) Theorem \ref{dowpearlprecise}, there exists a comeager subset of $X^\omega$ at every point of which $f$ acts as a (full) permutation of the coordinates. Furthermore, this permutation varies continuously.

The following example, which is our main result, shows that homeomorphisms of this kind are the only ones that can be constructed in $\ZFC$. So, in a sense, Theorem \ref{dowpearlprecise} is sharp. In Section 5, we will see that Theorem \ref{main} is also relevant to an open problem of Terada.
\begin{theorem}\label{main}
It is consistent that there exists a zero-dimensional separable metrizable space $Z$ satisfying the following conditions:
\begin{enumerate}
\item\label{baire} $Z$ is Baire (hence $Z^\omega$ is Baire),
\item\label{permutation} For every homeomorphism $f:Z^\omega\longrightarrow Z^\omega$ there exists a comeager subset $C$ of $Z^\omega$ such that for every $z\in C$ there exists a bijection $h_z:\omega\longrightarrow\omega$ such that $f(z)(i)=z(h_z(i))$ for all $i\in\omega$,
\item\label{continuous} The function $h:C\longrightarrow\omega^\omega$ defined by $h(z)=h_z$ is continuous.
\end{enumerate}
\end{theorem}
\begin{proof}
Let $M$ be a countable transitive model of $\ZFC$. Define $\PPP=\Fin(\omega_1\times\omega,2)$. Let $G$ be a $\PPP$-generic filter over $M$. In the forcing extension $M[G]$, we will denote by $c_\alpha:\omega\longrightarrow\omega$ for $\alpha\in\omega_1$ the $\alpha$-th Cohen real, that is $c_\alpha(j)=(\bigcup G)(\alpha,j)$ for each $\alpha\in\omega_1$ and $j\in\omega$. We claim that $Z=\{c_\alpha:\alpha\in\omega_1\}$ has the desired properties.

The fact that $Z$ is Baire follows from Corollary \ref{nowheremeager} and Proposition \ref{nowheremeagerequivalence}. By Theorem \ref{baireproduct}, this implies that $Z^\omega$ is Baire as well.
Condition $(\ref{permutation})$ follows immediately from Lemmas \ref{existsgamma} and \ref{existsgammaimpliespermutation}, while Condition $(\ref{continuous})$ is proved in Section 4.
\end{proof}

We remark that our initial approach to Theorem \ref{main} was to construct a space satisfying Conditions $(\ref{baire})$ and $(\ref{permutation})$ by a transfinite recursion of length $\mathfrak{c}$, using the assumption $\mathsf{cov}(\mathcal{M})=\mathfrak{c}$ to make sure that an analogue of Lemma \ref{existsgamma} would hold.\footnote{\,Recall that $\mathsf{cov}(\mathcal{M})$ is the minimum size of a collection of meager subsets of $2^\omega$ whose union is $2^\omega$.}
Subsequently, we realized that the set of Cohen reals has the same properties, and that it satisfies Condition $(\ref{continuous})$ as well. However, we do not know the answer to the following question.

\begin{question}
Is it possible to construct in $\ZFC$ a space $Z$ as in Theorem \ref{main}?
\end{question}

\section{Preliminaries and notation}

We will assume familiarity with the basic theory of forcing and Borel codes (see for example \cite{kunen} and \cite{jech}). We will also assume familiarity with Baire category. Our references for general topology will be \cite{engelking} and \cite{vanmill}.

We will often be dealing with $\omega$-th powers of subspaces of $2^\omega$. Therefore, for simplicity, we will identify an element $x\in (2^\omega)^\omega$ with an element of $2^{\omega\times\omega}$ by setting $x(i,j)=x(i)(j)$ for $i,j\in\omega$. Given sets $I$ and $J$, we will denote by $\Fin(I,J)$ the set of functions $s$ such that $\dom(s)$ is a finite subset of $I$ and $\ran(s)$ is a (finite) subset of $J$. Given $s\in\Fin(\omega,2)$, we will use the notation $[s]=\{x\in 2^\omega:s\subseteq x\}$. Similarly, given $s\in\Fin(\omega\times\omega,2)$, we will use the notation $[s]=\{x\in 2^{\omega\times\omega}:s\subseteq x\}$.

We will be freely using the following three results. We leave to the reader the proofs of the first two. For a proof of the third, see \cite[Theorem 3]{oxtoby}. Recall the following definitions.
A subset $S$ of a space $X$ is \emph{nowhere meager} if $S\cap U$ is non-meager in $X$ for every non-empty open subset $U$ of $X$.
A \emph{pseudobase} for a space $X$ is a collection $\BB$ consisting of non-empty open subsets of $X$ such that for every non-empty open subset $U$ of $X$ there exists $B\in\BB$ such that $B\subseteq U$.
\begin{proposition}\label{meagersubspace}
Let $Y$ be a dense subspace of the space $X$. If $S$ is meager in $X$ then $S\cap Y$ is meager in $Y$.
\end{proposition}
\begin{proposition}\label{nowheremeagerequivalence}
Let $X$ be a space and $S\subseteq X$. Then $S$ is nowhere meager in $X$ if and only if $S$ is dense in $X$ and Baire as a subspace of $X$.
\end{proposition}
\begin{theorem}[Oxtoby]\label{baireproduct}
The product of any family of Baire spaces, each of which has a countable pseudobase, is a Baire space.
\end{theorem}

The following lemma is essentially \cite[(2.1)]{oxtoby}, and we will use it in the proof of Lemma \ref{existsgamma}. Given spaces $X$, $Y$, a point $x\in X$ and $S\subseteq X\times Y$, we will use the notation $S[x]=\{y\in Y:(x,y)\in S\}$.
\begin{lemma}[Oxtoby]\label{oxtobylemma}
Let $X$ and $Y$ be spaces such that $Y$ has a countable pseudobase. Assume that $W$ is a dense $\Gd$ subset of $X\times Y$. Then there exists a comeager subset $C$ of $X$ such that $W[x]$ is a dense $\Gd$ in $Y$ whenever $x\in C$.	
\end{lemma}
\begin{proof}
Let $U_n$ for $n\in\omega$ be open dense subsets of $X\times Y$ such that $W=\bigcap_{n\in\omega}U_n$. We will construct a comeager subset $C$ of $X$ such that $U_n[x]$ is open dense in $Y$ for every $x\in C$ and $n\in\omega$. Since $W[x]=\bigcap_{n\in\omega}(U_n[x])$ for all $x\in X$, this will conclude the proof.

Let $\{B_m:m\in\omega\}$ be a pseudobase for $Y$. Let $\pi:X\times Y\longrightarrow X$ be the projection on the first coordinate. Define $V_{m,n}=\pi[(X\times B_m)\cap U_n]$ for $m,n\in\omega$. Observe that each $V_{m,n}$ is an open dense subset of $X$. It is straightforward to check that $C=\bigcap_{m,n\in\omega}V_{m,n}$ is the desired comeager set.
\end{proof}

For the proofs of the following two classical results, see \cite[Theorem 4.3.20]{engelking} and \cite[Theorem 4.3.21]{engelking}.
\begin{lemma}[Lavrentieff]\label{lavrentieffcontinuous}
Let $X$ and $Y$ be spaces, with $Y$ completely metrizable. Assume that $f:A\longrightarrow Y$ is continuous, where $A\subseteq X$. Then there exists a $\Gd$ subset $S$ of $X$ and a continuous function $g:S\longrightarrow Y$ such that $f\subseteq g$.
\end{lemma}
\begin{lemma}[Lavrentieff]\label{lavrentieffhomeomorphism}
Let $X$ and $Y$ be completely metrizable spaces. Assume that $f:A\longrightarrow B$ is a homeomorphism, where $A\subseteq X$ and $B\subseteq Y$. Then there exist $\Gd$ subsets $S$ of $X$ and $T$ of $Y$, and a homeomorphism $g:S\longrightarrow T$ such that $f\subseteq g$.
\end{lemma}

We will always denote by $M$ a countable transitive model of (a sufficiently large fragment) of $\ZFC$.
Given a $\Fin(I\times\omega,2)$-generic filter $G$ over $M$, in the forcing extension $M[G]$, we will denote by $c_\alpha:\omega\longrightarrow 2$ for $\alpha\in I$ the $\alpha$-th Cohen real, that is $c_\alpha(j)=(\bigcup G)(\alpha,j)$ for each $\alpha\in I$ and $j\in\omega$.
The most important case will be when $I=\omega_1$. In fact, throughout this paper, we will use the following notation: 
\begin{itemize}
\item $\PPP=\Fin(\omega_1\times\omega,2)$,
\item $Z=\{c_\alpha:\alpha\in\omega_1\}$,
\item $Z_\gamma=\{c_\alpha:\alpha\in\gamma\}$ for $\gamma\in\omega_1$,
\item $M_\gamma=M[G\cap\Fin(\gamma\times\omega,2)]$ for $\gamma\in\omega_1$,
\item $M[c_\alpha:\alpha\in J]=M[G\cap\Fin(J\times\omega,2)]$ for $J\subseteq\omega_1$.
\end{itemize}

The following two results are well-known. We will not give the proof of the first, as it is just a simpler version of the proof of Lemma \ref{getintocomeagerinfinite}.
\begin{proposition}\label{getintocomeager}
Let $W$ be a dense $\Gd$ subset of $2^\omega$ coded in $M$.
Let $I\in M$ be a non-empty set, and force with $\Fin(I\times\omega,2)$. Let $\alpha\in I$. Then $\mathbb{1}\Vdash\text{``}c_\alpha\in W\text{''}$.
\end{proposition}
\begin{corollary}\label{nowheremeager}
The set $Z$ is nowhere meager in $2^\omega$. 
\end{corollary}
\begin{proof}
Assume, in order to get a contradiction, that $Z\cap [s]$ is meager in $2^\omega$ for some $s\in\Fin(\omega,2)$. Let $W$ be a dense $\Gd$ subset of $2^\omega$ such that $W\subseteq (2^\omega\setminus [s])\cup ([s]\setminus Z)$. Fix $\gamma\in\omega_1$ such that $W$ is coded in $M_\gamma$. By applying Proposition \ref{getintocomeager} with $M=M_\gamma$ and $I=\omega_1\setminus\gamma$, one sees that $Z\cap [s]\subseteq Z_\gamma$. So let $p\in\PPP$ be such that
$$
p\Vdash\text{``}c_\alpha\notin [s]\text{ whenever }\gamma\leq\alpha<\omega_1\text{''}.
$$
Now fix $\delta\geq\gamma$ such that $(\{\delta\}\times\omega)\cap\dom(p)=\varnothing$, then define $q\in\PPP$ as follows:
\begin{itemize}
\item $\dom(q)=\dom(p)\cup(\{\delta\}\times\dom(s))$,
\item $q(\delta,j)=s(j)$ for every $j\in\dom(s)$,
\item $q(\alpha,j)=p(\alpha,j)$ for every $(\alpha,j)\in\dom(p)$.
\end{itemize}
It is clear that $q\leq p$. On the other hand, $q\Vdash\text{``}c_\delta\in [s]\text{''}$, which is contradiction.
\end{proof}

The following lemma is the ``$\omega$-th power'' of Proposition \ref{getintocomeager}, and it will be needed in the proof of Lemma \ref{existsgamma}.
\begin{lemma}\label{getintocomeagerinfinite}
Let $W$ be a dense $\Gd$ subset of $2^{\omega\times\omega}$ coded in $M$. Let $I\in M$ be an infinite set, and force with $\Fin(I\times\omega,2)$. Let $\langle\alpha_i:i\in\omega\rangle\in M$ be an injective sequence of elements of $I$. Then $\mathbb{1}\Vdash\text{``}\langle c_{\alpha_i}:i\in\omega\rangle\in W\text{''}$.
\end{lemma}
\begin{proof}
Let $U_n$ for $n\in\omega$ be dense open subsets of $2^{\omega\times\omega}$ such that the sequence of their codes belongs to $M$ and $W=\bigcap_{n\in\omega}U_n$. Assume, in order to get a contradiction, that there exists $p\in\Fin(I\times\omega,2)$ and $n\in\omega$ such that $p\Vdash\text{``}\langle c_{\alpha_i}:i\in\omega\rangle\notin U_n\text{''}$.
Define $s\in\Fin(\omega\times\omega,2)$ as follows:
\begin{itemize}
\item $\dom(s)=\{(i,j)\in\omega\times\omega:(\alpha_i,j)\in\dom(p)\}$,
\item $s(i,j)=p(\alpha_i,j)$ for every $(i,j)\in\dom(s)$.
\end{itemize}
Since $U_n$ is open dense, it is possible to find $t\in\Fin(\omega\times\omega,2)$ such that $s\subseteq t$ and $[t]\subseteq U_n$.
Finally, define $q\in\Fin(I\times\omega,2)$ as follows:
\begin{itemize}
\item $\dom(q)=\dom(p)\cup\{(\alpha_i,j):(i,j)\in\dom(t)\setminus\dom(s)\}$,
\item $q(\alpha_i,j)=t(i,j)$ for every $(i,j)\in\dom(t)\setminus\dom(s)$,
\item $q(\alpha,j)=p(\alpha,j)$ for every $(\alpha,j)\in\dom(p)$.
\end{itemize}
It is clear that $q\leq p$. On the other hand, $q\Vdash\text{``}\langle c_{\alpha_i}:i\in\omega\rangle\in [t]\text{''}$, which is a contradiction.
\end{proof}

Recall that a space $X$ is \emph{rigid} if the only homeomorphism $f:X\longrightarrow X$ is the identity (see \cite{medinivanmillzdomskyy} for several references on this topic). The following proposition will not be needed in the rest of the paper, but we decided to keep it, as its proof is particularly simple and it provides a good warm-up for the case of $Z^\omega$. In fact, as our discussion in Section 1 shows, Theorem \ref{main} can be seen as a ``rigidity-type'' result. 
\begin{proposition}
The space $Z$ is rigid. 
\end{proposition}
\begin{proof}
Let $f:Z\longrightarrow Z$ be a homeomorphism. We will show that there exists $\gamma\in\omega_1$ such that $f(c_\alpha)=c_\alpha$ for every $\alpha\geq\gamma$. In particular, by Corollary \ref{nowheremeager}, the function $f$ will be the identity on a dense subset of $Z$. Since $f$ is continuous, this will conclude the proof.

By Lemma \ref{lavrentieffhomeomorphism}, there exist $\Gd$ subsets $S$ and $T$ of $2^\omega$, and a homeomorphism $g:S\longrightarrow T$ such that $f\subseteq g$.
Notice that $g$ is a closed subset of $S\times T$, hence a Borel (in fact, $\Gd$) subset of $2^\omega\times 2^\omega$. So we can fix $\gamma\in\omega_1$ such that $g$ is coded in $M_\gamma$. We will show that $\gamma$ is as desired.

First we claim that $g\re Z_\gamma:Z_\gamma\longrightarrow Z_\gamma$ is a bijection. If $\alpha <\gamma$ then $c_\alpha\in M_\gamma$, hence $g(c_\alpha)\in M_\gamma$, while on the other hand $g(c_\alpha)=f(c_\alpha)\in Z$. Since $M_\gamma\cap Z=Z_\gamma$, this shows that $g[Z_\gamma]\subseteq Z_\gamma$. But $g^{-1}$ is also coded in $M_\gamma$, and a similar argument shows that $g^{-1}[Z_\gamma]\subseteq Z_\gamma$. This concludes the proof of the claim.

Now pick $\alpha\geq\gamma$. Since $c_\alpha\in M_\gamma[c_\alpha]$, it is clear that $g(c_\alpha)\in M_\gamma[c_\alpha]$. On the other hand $g(c_\alpha)=f(c_\alpha)\in Z$, hence $g(c_\alpha)=c_\beta$ for some $\beta\in\gamma\cup\{\alpha\}$. If we had $\beta\in\gamma$ then, by our claim, the injectivity of $g$ would be contradicted. Therefore, we must have $\beta=\alpha$.
\end{proof}

\section{Every homeomorphism is a permutation of the coordinates on a comeager set}

Given a function $f:Z^\omega\longrightarrow Z^\omega$ and $\gamma\in\omega_1$, define
$$
A(f,\gamma)=\{x\in Z^\omega:f(x)\in (Z_\gamma\cup\{x(i):i\in\omega\})^\omega\}.
$$
\begin{lemma}\label{existsgamma}
Let $f:Z^\omega\longrightarrow Z^\omega$ be a continuous function. Then there exists $\gamma\in\omega_1$ such that $A(f,\gamma)$ is comeager in $Z^\omega$.
\end{lemma}
\begin{proof}
By Lemma \ref{lavrentieffcontinuous}, there exists a $\Gd$ subset $S$ of $2^{\omega\times\omega}$ and a continuous function $g:S\longrightarrow 2^{\omega\times\omega}$ such that $f\subseteq g$. Fix $\gamma\in\omega_1$ such that $g$ is coded in $M_\gamma$. We claim that $\gamma$ is as desired.

In fact, we will show that the set
$$
A=\{x\in S:g(x)\in (Z_\gamma\cup\{x(i):i\in\omega\})^\omega\}
$$
is comeager in $2^{\omega\times\omega}$. Assume, in order to get a contradiction, that $A$ is non-comeager. As $Z_\gamma$ is countable and $g$ is continuous, it is easy to check that $A$ is Borel. In particular, it has the Baire property (see \cite[Corollary A.13.9]{vanmill}), hence there exists $\ell\in\omega$ and $s:\ell\times\ell\longrightarrow 2$ such that $A\cap [s]$ is meager.

Define $X=2^{\ell\times\omega}$ and $Y=2^{(\omega\setminus\ell)\times\omega}$. Identify $X\times Y$ with $2^{\omega\times\omega}$, and let $W$ be a dense $\Gd$ subset of $X\times Y$ such that $W\subseteq (2^{\omega\times\omega}\setminus [s])\cup([s]\setminus A)$. Since $S$ is comeager in $2^{\omega\times\omega}$, we can also assume that $W\subseteq S$. An application of Lemma \ref{oxtobylemma} yields a comeager subset $C$ of $X$ such that $W[x]$ is a dense $\Gd$ subset of $Y$ whenever $x\in C$. Since $Z^\ell$ is nowhere meager in $2^{\ell\times\omega}$ by Corollary \ref{nowheremeager}, we can fix a sequence $\langle\alpha_i:i\in\ell\rangle$ with $\alpha_i<\omega_1$ for each $i\in\ell$, such that $s\subseteq x$ and $x\in C$, where $x=\langle c_{\alpha_i}:i\in\ell\rangle$. Let $\delta\geq\gamma$ be such that $W$ is coded in $M_\delta$, and notice that $W[x]$ is coded in $M[c_\alpha:\alpha\in\delta\cup\{\alpha_i:i\in\ell\}]$. Therefore, if we fix $\zeta>\supi(\delta\cup\{\alpha_i:i\in\ell\})$ and define $\alpha_i=\zeta+i$ for $i\in\omega\setminus\ell$, an application of Lemma \ref{getintocomeagerinfinite} will show that $y\in W[x]$, where $y=\langle c_{\alpha_i}:i\in\omega\setminus\ell\rangle$. It follows that $z\in W$, where $z=\langle c_{\alpha_i}:i\in\omega\rangle$. Furthermore, the fact that $s\subseteq x$ shows that $z\in [s]$, hence $z\notin A$.

The only thing to observe about the sequence $\langle\alpha_i:i\in\omega\rangle$ is that it belongs to $M$ (thanks to its simple definition), hence $z\in M[c_{\alpha_i}:i\in\omega]\subseteq M_\gamma[c_{\alpha_i}:i\in\omega]$. It follows that $g(z)\in M_\gamma[c_{\alpha_i}:i\in\omega]$. On the other hand, $M_\gamma[c_{\alpha_i}:i\in\omega]\cap (Z^\omega)\subseteq (Z_\gamma\cup\{c_{\alpha_i}:i\in\omega\})^\omega$, which contradicts the fact that $z\notin A$.
\end{proof}

\begin{lemma}\label{existsgammaimpliespermutation}
Let $f:Z^\omega\longrightarrow Z^\omega$ be a homeomorphism. Assume that $\gamma\in\omega_1$ is such that $A(f,\gamma)$ and $A(f^{-1},\gamma)$ are both comeager in $Z^\omega$. Then there exists a comeager subset $C$ of $Z^\omega$ such that for every $z\in C$ there exists a permutation $h_z:\omega\longrightarrow\omega$ such that $f(z)(i)=z(h_z(i))$ for every $i\in\omega$.
\end{lemma}
\begin{proof}
Define
$$
B=A(f,\gamma)\cap A(f^{-1},\gamma)\cap\bigcap_{i\in\omega}\{z\in Z^\omega:z(i)\notin Z_\gamma\}\cap\bigcap_{i\in\omega}\{z\in Z^\omega:f(z)(i)\notin Z_\gamma\},
$$
and notice that $B$ is comeager in $Z^\omega$, because $Z_\gamma$ is countable and $f$ is a homeomorphism. Furthermore, the definition of $A(f,\gamma)$ easily implies that 
$$
\{f(z),f^{-1}(z)\}\subseteq\{z(i):i\in\omega\}^\omega
$$
for every $z\in B$.

Let $\Delta=\{z\in Z^\omega:z(i)=z(j)\text{ for some }i\neq j\}$. Then, for every $z\in B\setminus\Delta$, there will be a unique function $h_z:\omega\longrightarrow\omega$ such that $f(z)(i)=z(h_z(i))$ for every $i\in\omega$. Similarly, for every $z\in B\setminus\Delta$, there will be a unique function $k_z:\omega\longrightarrow\omega$ such that $f^{-1}(z)(i)=z(k_z(i))$ for every $i\in\omega$.

Finally, define
$$
C=(B\setminus\Delta)\cap f^{-1}[B\setminus\Delta],
$$
and observe that $C$ is comeager in $Z^\omega$. For any fixed $z\in C$, it is clear that $\{z,f(z)\}\subseteq B\setminus\Delta$, and it is straightforward to check that $k_{f(z)}$ is the inverse function of $h_z$. In particular, $h_z$ is a bijection, which concludes the proof.
\end{proof}

\section{The permutation varies continuously}

In this section, we will prove that Condition $(\ref{continuous})$ of Theorem \ref{main} holds. We will use the same notation as in the previous section.
More precisely, we assume that a homeomorphism $f:Z^\omega\longrightarrow Z^\omega$ is given, and let $g:S\longrightarrow T$ be a homeomorphism between $\Gd$ subsets of $2^{\omega\times\omega}$ such that $f\subseteq g$, whose existence is guaranteed by Lemma \ref{lavrentieffhomeomorphism}.
Then, we let $\gamma\in\omega_1$ be such that $g$ (hence $g^{-1}$ as well) are coded in $M_\gamma$. As in the proof of Lemma \ref{existsgamma}, this will guarantee that $A(f,\gamma)$ and $A(f^{-1},\gamma)$ are comeager in $Z^\omega$.
So, defining $C$ as in the proof of Lemma \ref{existsgammaimpliespermutation} will guarantee that Condition $(\ref{permutation})$ of Theorem \ref{main} holds.

We will show that for all $n\in\omega$ and $x\in C$ there exists $s\in\Fin(\omega\times\omega,2)$ such that $s\subseteq x$ and $h_y(n)=h_x(n)$ whenever $y\in C\cap [s]$. So fix $n\in\omega$ and $x\in C$, say $x(i)=c_{\alpha_i}$ for $i\in\omega$, where $\gamma\leq\alpha_i<\omega_1$ for each $i$. From now on we will treat $M_\gamma$ as our ground model. So $M[G]=M_\gamma[H]$ for some $\QQQ$-generic filter $H$ over $M_\gamma$, where $\QQQ=\Fin((\omega_1\setminus\gamma)\times\omega,2)$.

Let $p\in\QQQ$ and $m\in\omega$ be such that 
$$
p\Vdash\text{``}g(x)(n)=x(m)\text{''}.
$$
Now define $s\in\Fin(\omega\times\omega,2)$ as follows:
\begin{itemize}
\item $\dom(s)=\{(i,j)\in\omega\times\omega:(\alpha_i,j)\in\dom(p)\}$,
\item $s(i,j)=p(\alpha_i,j)$ for every $(i,j)\in\dom(s)$.
\end{itemize}
We will show that $h_y(n)=m$ whenever $y\in C\cap [s]$.
In order to get a contradiction, assume that there exist $y\in C\cap [s]$ and $k\in\omega$ such that $g(y)(n,k)=f(y)(n,k)\neq y(m,k)$. Let $\varepsilon=y(m,k)$. Then, by the continuity of $g$, there exists $\ell\in\omega$ such that $g(z)(n,k)\neq\varepsilon$ whenever $z\in [y\re (\ell\times\ell)]\cap S$. Without loss of generality, assume that $\{m,k\}\subseteq\ell$ and $\dom(s)\subseteq\ell\times\ell$. Let $t=y\re (\ell\times\ell)$, and observe that $s\subseteq t$.

Next, we claim that
\begin{itemize}
\item[$\circledast$]$\mathbb{1}_{\QQQ}\Vdash\text{``}g(z)(n,k)\neq \varepsilon$ whenever $z\in [t]\cap S\text{''}$.
\end{itemize}
Let $\varphi$ be the formula in quotes, and let $N$ be an arbitrary generic extension of $M_\gamma$ obtained by forcing with $\QQQ$. In order to conclude the proof of the claim, it will be enough to show that $\varphi$ holds in $N$. Obviously $M_\gamma[H]\vDash\varphi$, and in particular $g(z)(n,k)=1-\varepsilon$ for every $z\in [t]\cap S\cap M_\gamma$.
So the set $[t]\cap S\cap M_\gamma$ witnesses that $N\vDash\text{``}g(z)(n,k)=1-\varepsilon\text{ for a dense set of }z\in [t]\cap S\text{''}$. Since $g$ is continuous, it follows that $N\vDash\varphi$.

Finally, define $q\in\QQQ$ as follows:
\begin{itemize}
\item $\dom(q)=\dom(p)\cup\{(\alpha_i,j):(i,j)\in\ell\times\ell\}$,
\item $q(\alpha_i,j)=t(i,j)$ for every $(i,j)\in (\ell\times\ell)$ such that $(\alpha_i,j)\notin\dom(p)$,
\item $q(\alpha,j)=p(\alpha,j)$ for every $(\alpha,j)\in\dom(p)$.
\end{itemize}
It is clear that $q\leq p$. Furthermore, using the fact that $s\subseteq t$, it is easy to check that $q\Vdash\text{``}x\re (\ell\times\ell)=t\text{''}$. On the other hand, since $q\leq p$, we see that $q\Vdash\text{``}g(x)(n,k)=x(m,k)\text{''}$. This contradicts $\circledast$.

\section{A problem of Terada}

A space $X$ is \emph{h-homogeneous} if every non-empty clopen subspace of $X$ is homeomorphic to $X$. This notion has been studied by several authors, both ``instrumentally'' and for its own sake (see for example the references in \cite{medinih}). The following proposition is well-known (see for example \cite[Proposition 3.32 and Figure 3.33]{medinit}), and it explains why h-homogeneous spaces are sometimes called \emph{strongly homogeneous}.

\begin{proposition}\label{folklore}
Let $X$ be a first-countable zero-dimensional space. If $X$ is h-homogeneous then $X$ is homogeneous.
\end{proposition}

The following question from \cite{terada} remains open (even in the separable metrizable case), and it was the original motivation for our research. In fact, Theorem \ref{main} was born out of an attempt to construct a counterexample to it. Notice that, by Proposition \ref{folklore}, an affirmative answer to Question \ref{teradaquestion} would yield a strengthening of Theorem \ref{dowpearlmain}.
\begin{question}[Terada]\label{teradaquestion}
Is $X^\omega$ h-homogeneous for every zero-dimensional first-countable space $X$?
\end{question}

Next, we list a few partial results on Question \ref{teradaquestion}. The following theorem is due independently to van Engelen and Medvedev (see \cite[Theorems 4.2 and 4.4]{vanengelen} or \cite[proof of Theorem 25]{medvedev}).\footnote{\,Medvedev assumes $\Ind(X)=0$, but it is well-known that $\dime(X)=\Ind(X)$ for every metrizable space $X$ (see \cite[Theorem 7.3.2]{engelking}).}
Recall that the assumption $\dime(X)=0$ on a metrizable space $X$ is in general stronger than the assumption that $X$ is zero-dimensional (see \cite{roy}). However, these two assumptions are equivalent if $X$ is also separable (see \cite[Theorem 7.3.3]{engelking}).
\begin{theorem}[van Engelen, Medvedev]\label{meagerdensecomplete}
	Let $X$ be a metrizable space such that $\dime(X)=0$. Assume that either $X$ is meager or $X$ has a dense completely metrizable subspace. Then $X^\omega$ is h-homogeneous.
\end{theorem}
\begin{corollary}
	Assume that $X$ belongs to the $\sigma$-algebra generated by the analytic subsets of $2^\omega$. Then $X^\omega$ is h-homogeneous.
\end{corollary}
\begin{proof}
By \cite[Propositions 3.4 and 3.3]{medinic}, it follows that either $X$ has a dense completely metrizable subspace or $X$ is not Baire. In the first case, $X^\omega$ will have a completely metrizable dense subspace as well.
In the second case, it is easy to see that $X^\omega$ will be meager (see for example \cite[proof of Proposition 4.4]{medinic}). The proof is concluded by observing that $(X^\omega)^\omega$ is homeomorphic to $X^\omega$.
\end{proof}

The following result, which first appeared as \cite[Corollary 29]{medinih}, shows that the additional requirements in Theorem \ref{meagerdensecomplete} are not necessary, provided that $X$ is ``big'' enough.
\begin{theorem}[Medini]
	Let $X$ be a metrizable space such that $\dime(X)=0$. Assume that $X$ is non-separable. Then $X^\omega$ is h-homogeneous.
\end{theorem}
The following result is a particular case of \cite[Theorem 18]{medinih}, which generalizes results of Motorov and Terada.
\begin{theorem}[Medini]
	Let $X$ be a Tychonoff space such that the isolated points are dense. Then $X^\omega$ is h-homogeneous.
\end{theorem}
The following result follows immediately from \cite[Proposition 24 and Lemma 22]{medinih}. Recall that a space $X$ is \emph{divisible} by $2$ if there exists a space $Y$ such that $X=Y\times 2$, where $2$ is the discrete space with two elements.
\begin{theorem}[Medini]\label{divisible}
	Let $X$ be a zero-dimensional first-countable space containing at least two points. Then $X^\omega$ is h-homogeneous if and only if $X^\omega$ is divisible by $2$.
\end{theorem}

An interesting consequence of Theorem \ref{divisible} is that, in order to answer Question \ref{teradaquestion} in the affirmative, it would be enough to exhibit a clopen subset $C$ of $X^\omega$ and a homeomorphism $f:X^\omega\longrightarrow X^\omega$ such that $f[C]=X^\omega\setminus C$ (and $f[X^\omega\setminus C]=C$).
While Theorem \ref{main} does not resolve Question \ref{teradaquestion}, it does show that (if one wants to give a general construction) the homeomorphism $f$ would have to be of the same kind as those constructed in \cite{dowpearl} and \cite{lawrence}.

\end{document}